\documentclass[a4paper,11pt,reqno]{amsart}
\usepackage{a4wide}

\usepackage{enumerate}
\usepackage{amssymb, amsmath}
\usepackage{mathrsfs}
\usepackage{amscd}
\usepackage[active]{srcltx}
\usepackage{verbatim}
\usepackage[colorlinks,linkcolor={black},citecolor={blue},urlcolor={black}]{hyperref}

\let\emptyset \undefined
\newsymbol\emptyset    203F

\theoremstyle{plain}
\newtheorem{theorem}{Theorem}[section]
\newtheorem{corollary}[theorem]{Corollary}
\newtheorem{lemma}[theorem]{Lemma}
\newtheorem{proposition}[theorem]{Proposition}

\newtheorem{definition}[theorem]{Definition}
\newtheorem{assumption}[theorem]{Assumption}

\newtheorem*{definition*}{Definition}

\theoremstyle{remark}
\newtheorem{remark}[theorem]{Remark}

\newtheorem*{claim*}{Claim}
\newtheorem*{remark*}{Remark}
\newtheorem*{example*}{Example}
\newtheorem*{notation*}{Notation}

\newcommand{\PTN}{\mathcal{P}(\TN)}

\newcommand{\cWN}{\mathcal{W}_N}
\numberwithin{equation}{section}

\def\Z{{\mathbb Z}}
\def\R{{\mathbb R}}

\def\T{{\mathbb T}}
\newcommand{\tWN}{\widetilde{\mathcal{W}}_N}

\newcommand{\one}{{{\bf 1}}}
\newcommand{\TN}{{{\mathbb{T}_N^d}}}

\newcommand{\eps}{\varepsilon}
\renewcommand{\phi}{\varphi}

\newcommand{\PPhi}{\boldsymbol{\Phi}}
\newcommand{\hDelta}{\hat\Delta}
\newcommand{\dd}{\; \mathrm{d}}

\DeclareMathOperator{\Ran}{Ran}

\DeclareMathOperator{\Ker}{Ker}
\DeclareMathOperator{\lin}{lin}

\newcommand{\ip}[1]{\langle {#1}\rangle}
\newcommand{\bip}[1]{\big\langle {#1}\big\rangle}

\DeclareMathOperator{\arctanh}{arctanh}

\DeclareMathOperator{\grad}{grad}

\DeclareMathOperator{\Hess}{Hess}

\newcommand{\ddt}{\frac{\mathrm{d}}{\mathrm{d}t}}

\newcommand{\ddtt}{\frac{\mathrm{d^2}}{\mathrm{d}t^2}}

\newcommand{\cH}{\mathcal{H}}
\newcommand{\cB}{\mathcal{B}}

\newcommand{\cW}{\mathcal{W}}
\newcommand{\cI}{\mathcal{I}}

\newcommand{\sH}{\mathscr{H}}

\newcommand{\cX}{\mathcal{X}}
\newcommand{\cT}{\mathbf{T}}

\newcommand{\cA}{\mathcal{A}}

\newcommand{\CE}{\mathcal{CE}}

\newcommand{\cF}{\mathcal{F}}

\newcommand{\cP}{\mathscr{P}}
\newcommand{\PX}{\cP(\cX)}
\newcommand{\PXs}{\cP_*(\cX)}
\newcommand{\hrho}{\hat\rho}

\renewcommand{\tilde}{\widetilde}

\begin{document}

\title
[Discrete porous medium equations]
{Gradient flow structures for discrete porous medium equations}

\author{Matthias Erbar}
\author{Jan Maas}
\address{
University of Bonn\\
Institute for Applied Mathematics\\
Endenicher Allee 60\\
53115 Bonn\\
Germany}
\email{erbar@iam.uni-bonn.de}
\email{maas@uni-bonn.de}


 \begin{abstract}
   We consider discrete porous medium equations of the form
   $\partial_t\rho_t = \Delta \phi(\rho_t)$, where $\Delta$ is the
   generator of a reversible continuous time Markov chain on a finite
   set $\cX$, and $\phi$ is an increasing function. We show that these
   equations arise as gradient flows of certain entropy functionals
   with respect to suitable non-local transportation metrics.  This
   may be seen as a discrete analogue of the Wasserstein gradient flow
   structure for porous medium equations in $\R^n$ discovered by Otto.
   We present a one-dimensional counterexample to geodesic convexity
   and discuss Gromov-Hausdorff convergence to the Wasserstein metric.
 \end{abstract} 


\maketitle
 

\section{Introduction}
\label{sec:intro}

Recently it has been shown that discretisations of heat equations and
Fokker-Planck equations can be formulated as gradient flows of the
entropy with respect to a non-local transportation metric $\cW$ on the
space of probability measures. Results in this direction have been
obtained independently in various settings, including Fokker-Planck
equations on graphs \cite{CHLZ11}, reversible Markov chains
\cite{Ma11}, and systems of reaction-diffusion equations
\cite{Mie11a}. Related gradient flow structures have been found for
fractional heat equations \cite{Erb12} and quantum mechanical
evolution equations \cite{CaMa12,Mie12}.

The above-mentioned results can be regarded as discrete counterparts
to the by now classical result of Jordan, Kinderlehrer and Otto
\cite{JKO98}, who showed that Fokker-Planck equations on $\R^n$ are
gradient flows of the entropy with respect to the $L^2$-Wasserstein
metric on the space of probability measures. In this continuous
setting there are many other interesting partial differential
equations which admit a formulation as Wasserstein gradient
flow. Among them, one of the most prominent examples is the porous
medium equation, which has been identified as the Wasserstein gradient
flow of the R\'enyi entropy in Otto's seminal paper \cite{O01}.

It seems therefore natural to ask whether a similar gradient flow
structure exists for discrete versions of porous medium equations. In
this paper we show that suitable discretisations indeed admit a
gradient flow structure. The associated metrics turn out to be
variations on the metric $\cW$, which have already been studied in
\cite{Ma11}.

\subsection{The discrete setting}

In this paper we let $\cX$ be a finite set. We consider a matrix $Q :
\cX \times \cX \to \R$ satisfying $Q(x,y) \geq 0$ for all $x \neq y$
and $\sum_{y \in \cX} Q(x,y) = 0$. We assume that $Q$ is irreducible,
so that basic Markov chain theory implies the existence of a unique
invariant probability measure $\pi$ on $\cX$. We assume that $\pi$ is
reversible, i.e., the \emph{detailed balance equations}
\begin{align}\label{eq:detailed-balance-intro}
 Q(x,y)\pi(x) = Q(y,x)\pi(y)
\end{align}
hold for all $x, y \in \cX$. 
For a function $\psi : \cX \to \R$ we write
\begin{align}\label{eq:discrete-Laplace}
 \Delta \psi(x) := \sum_{y \in \cX} Q(x,y) \psi(y)
            = \sum_{y \in \cX} Q(x,y) (\psi(y) - \psi(x))\;.
\end{align}
The operator $\Delta$ is the generator of a continuous time Markov
semigroup $(P_t)_{t \geq 0}$. By the reversibility assumption, this
semigroup is selfadjoint on $L^2(\cX,\pi)$. Moreover, the set $\PX$ of
probability densities with respect to $\pi$, 
is invariant under the semigroup $(P_t)_{t \geq 0}$. 

We shall consider a class of non-local transportation metrics on the
space of probability densities $\PX$. The definition of such a metric \cite{Ma11}
requires a weight-function $\theta : [0,\infty) \times [0,\infty) \to
[0,\infty)$ satisfying appropriate conditions (see Definition
\ref{def:weight} below). For $\rho \in \PX$ and $x, y \in \cX$ we then
define $\hrho(x,y) = \theta(\rho(x), \rho(y))$. Inspired by the
Benamou-Brenier formula \cite{BB00}, we define for $\bar\rho_0,
\bar\rho_1 \in \PX$ the distance $\cW(\bar\rho_0, \bar\rho_1 )$ by
\begin{align*}
  \cW(\bar\rho_0, \bar\rho_1)^2 := 
  \inf_{\rho_\cdot, \psi_\cdot} \bigg\{ \frac12\int_0^1 \sum_{x,y\in \cX} (\psi(y)-\psi(x))^2\hrho(x,y) Q(x,y) \pi(x)\dd t\bigg\}\;,
\end{align*}
where the infimum runs over sufficiently regular curves $\rho : [0,1]
\to \PX$ and $\psi : [0,1] \times \cX \to \R$ satisfying the discrete
continuity equation
\begin{align*}
 \partial_t \rho_t(x) 
         + \sum_{y \in \cX}
         \big(\psi_t(y) - \psi_t(x)\big)\hrho_t(x,y) Q(x,y) = 0
\end{align*}
for all $x \in \cX$, with boundary conditions $\rho_0 = \bar\rho_0$
and $\rho_0 = \bar\rho_1$. Actually, $\cW$ is the Riemannian distance
induced by a Riemannian metric defined on the interior $\PXs$ of
$\PX$. 

Consider now the relative entropy functional $\cH : \PX \to \R$ defined by 
\begin{align*}
 \cH(\rho) := \sum_{x \in \cX} \rho(x) \log \rho(x) \;\pi(x)\;.
\end{align*}
 The following result has been shown in
\cite{CHLZ11,Ma11,Mie11a}: \emph{Solutions to the heat equation
  $\partial_t \rho_t = \Delta \rho_t$ are gradient flow trajectories of
  $\cH$ with respect to the metric $\cW$, provided that $\theta$ is
  the \emph{logarithmic mean} defined by
\begin{align}\label{eq:log-mean-intro}
 \theta(r,s) := \int_0^1 r^{1-\alpha} s^\alpha \dd \alpha\;.
\end{align}
}

\subsection{A gradient flow structure for discrete porous medium equations}
In this paper we shall consider non-linear generalisations of the
discrete heat equation. For suitable (see Assumption \ref{ass:phi})
strictly increasing functions $\phi : [0,\infty) \to \R$ we consider
the \emph{discrete porous medium type equation} given by
\begin{align}\label{eq:intro-pm}
  \partial_t \rho_t = \Delta \phi(\rho_t)\;,
\end{align}
where $\Delta$ denotes the discrete Laplacian defined in
\eqref{eq:discrete-Laplace}.  This equation can be analysed using
classical Hilbertian gradient flow techniques in a suitable ``discrete
$H^{-1}$-space'' (see Proposition \ref{prop:classical} below).

Here we are interested in more general gradient flow structures in the
spirit of the Wasserstein gradient flow structure for the porous
medium equation \cite{O01}.  For this purpose, for suitable (see
Assumption \ref{ass:f}) strictly convex functions $f : [0,\infty) \to
\R$, we consider the entropy functional $\cF : \PX \to \R$ defined by
\begin{align}\label{eq:intro-f-entropy}
 \cF(\rho) := \sum_{x \in \cX} f(\rho(x)) \pi(x)\;.
\end{align}

We shall prove the following result. 

\begin{theorem}\label{thm:intro-gradient-flow}
  Let $\phi$ and $f$ be as a above, and let $\cW$ denote the non-local
  transportation metric induced by
  \begin{align}\label{eq:theta-f-phi}
  \theta(r,s) = \frac{\phi(r) - \phi(s)}{f'(r) - f'(s)}\;.
\end{align}
Then the gradient flow equation of the entropy functional $\cF$ with
respect to $\cW$ is the discrete porous medium equation $\partial_t
\rho = \Delta \phi(\rho)$.
\end{theorem}

\begin{remark}\label{rem:examples}
  In the special case where $\phi(r) = r$ and $f(r) = r \log r$ we
  recover the result from \cite{CHLZ11,Ma11,Mie11a}. In this case we
  have $\theta(r,s) = \frac{r-s}{\log r - \log s}$, which coincides
  with \eqref{eq:log-mean-intro}.  The possibility of allowing more
  general functions $f$ has already been discussed in \cite{Ma11}.
\end{remark}

\begin{remark}\label{rem:power-mean}
  Of particular interest is the case where $\phi(r)=r^m$ and $f(r) =
  \frac{1}{m-1}r^m$ for some $0 < m \leq 2$. In this case the equation
  can be considered as a discrete porous medium equation (if $m > 1$)
  or fast diffusion equation (if $m < 1$). The functional $\cF$
  becomes a R\'enyi entropy $\cF_m$, so that Theorem
  \ref{thm:intro-gradient-flow} can be considered as a discrete
  analogue of the gradient flow structure obtained by Otto \cite{O01}.
  The expression for $\theta$ from \eqref{eq:theta-f-phi} becomes
\begin{align}\label{eq:power-mean}
\theta_m(r,s) = \frac{m-1}{m}\frac{r^m - s^m}{r^{m-1} - s^{m-1}}\;.
\end{align}
Several classical means are special cases of this expression.  In
fact, $\theta_2(r,s)= \frac{r+s}{2}$ is the arithmetic mean, $\lim_{m
  \to 1}\theta_m(r,s)$ is the logarithmic mean, and $\theta_{1/2}(r,s)
= \sqrt{rs}$ is the geometric mean.  Moreover, $\theta_{-1}(r,s) =
\frac{2rs}{r+s}$ is the harmonic mean, but we shall not consider
negative values of $m$ in the sequel.
\end{remark}

\begin{remark}\label{rem:reaction}
  For integer values of $m$, the porous medium equation takes the form
  of a chemical reaction equation, for which a gradient flow structure
  has been found in \cite{Mie11a}. However, there the driving
  functional is the relative entropy $\cH$, so that the associated
  weight function can be expressed in terms of logarithmic means. Here
  we also allow for different driving functionals, such as the R\'enyi
  entropy.
\end{remark}

\subsection{Geodesic convexity of entropy functionals}
Of crucial importance in the theory of Wasserstein gradient flows is
that the relevant entropy functionals exhibit good convexity
properties along Wasserstein geodesics. Recent work \cite{EM11,Mie11b}
shows that analogous properties hold for the relative entropy
functional $\cH$ in some discrete examples. In particular, $\cH$ turns
out to be convex along $\cW$-geodesics in one-dimensional discrete
Fokker-Planck equations \cite{Mie11b} as well as in heat equations on
$d$-dimensional square lattices in arbitrary dimension \cite{EM11}.

It thus seems natural to ask whether similar convexity properties hold
for the more general functionals $\cF$ along $\cW$-geodesics with the
appropriate choice of $\theta$ given by \eqref{eq:theta-f-phi}. In the
continous setting, it follows from a fundamental result by McCann
\cite{McC97} that the R\'enyi entropy $\cF_m(\rho) = \frac{1}{m-1}
\int_{\R^n} \rho^m(x) \dd x$ is displacement convex, i.e., convex
along $L^2$-Wasserstein geodesics in $\cP(\R^n)$, for $m \geq 1 -
\frac1n$.

However, in this paper we shall show that the discrete analogue fails
in general. We present a counterexample (Proposition
\ref{prop:counterexample}) in the case $m = 2$, which shows that even
in one dimension the R\'enyi entropy fails to be convex along
geodesics in the associated non-local transportation metric.

\begin{proposition}\label{prop:intro-counterexample}
  For $N \geq 6$, let $Q$ be the generator of simple random walk on
  the discrete circle $\cT_N=\Z/N\Z$. Let $\cW$ be the non-local
  transportation metric associated with the arithmetic mean
  $\theta_2$. Then the R\'enyi entropy $\cF_2$ is \emph{not} convex
  along $\cW$-geodesics.
\end{proposition}

It remains an interesting open question whether the R\'enyi entropy
$\cF_m$ is geodesically convex in some non-trivial discrete examples
for values of $m \neq 1$.

\subsection{Gromov-Hausdorff convergence of discrete transportation metrics}
Since the discrete transportation metrics discussed in this paper take
over the role of the $L^2$-Wasserstein metric in a discrete setting,
it seems natural to ask whether they converge to the $L^2$-Wasserstein
metric by a suitable limiting procedure. First results in this spirit
have been obtained in \cite{GM12}, where it was proved that discrete
transportation metrics $\cW_N$ associated with simple random walk on
the discrete torus $(\Z / N \Z)^d$ converge to the $L^2$-Wasserstein
metric $W_2$ over the torus $\mathbb{T}^d$. The weight function
considered in \cite{GM12} is the logarithmic mean. In the present
paper we observe that the result extends to more general weight
functions associated to porous medium equations.

\begin{theorem}\label{thm:GH-intro}
  Let $d \geq 1$, let $0 < m \leq 2$, and let $\cWN$ be the
  renormalised discrete transportation metric on $\PTN$ associated
  with simple random walk on $\TN$ and weight function $\theta_m$.
  Then the metric spaces $(\PTN,\cW_N)$ converge to
  $(\cP(\mathbb{T}^d),W_2)$ in the sense of Gromov-Hausdorff as $N \to
  \infty$.
\end{theorem}

Since this result can be obtained by a minor modification of the proof
in \cite{GM12}, we do not give a detailed proof here, but merely point
out the crucial properties which make the argument go through.

\section{Preliminaries on non-local transportation metrics}
\label{sec:prelim}
  
In this section we collect some preliminary results on non-local
transportation metrics. The presentation here is based on \cite{Ma11}.
  
In this paper, a reversible Markov chain consists of a triple $(\cX,
Q, \pi)$ where
\begin{itemize}
\item $\cX$ is a finite set;
\item $Q : \cX \times \cX \to \R$ is a $Q$-matrix, i.e., $Q(x,y) \geq
  0$ for all $x \neq y$ and $\sum_{y \in \cX} Q(x,y) = 0$ for all $x
  \in \cX$. Throughout this paper we assume that $Q$ is irreducible,
  i.e., for all $x,y \in \cX$ there exist $n \geq 1$ and $x_0, \ldots
  x_n \in \cX$ such that $x_0 = x$, $x_n = y$ and $Q(x_{i-1},x_i)> 0$
  for all $1 \leq i \leq n$;
\item $\pi$ is the associated stationary probability measure on $\cX$,
  which exists uniquely by basic Markov chain theory. We impose the
  standing assumption that $\pi$ is reversible, i.e., we assume that
  the \emph{detailed balance equations}
\begin{align}\label{eq:detailed-balance}
 Q(x,y)\pi(x) = Q(y,x)\pi(y)
\end{align}
are satisfied for all $x, y \in \cX$.
\end{itemize}

The set of probability densities with respect to $\pi$ will be
denoted by
\begin{align*}
 \PX := \Big\{ \, \rho : \cX \to \R \ | \ 
    \rho(x) \geq 0 \quad \forall x \in \cX\   \;; \ \sum_{x \in \cX} \rho(x)\pi(x)   = 1 \, \Big\}\;,
\end{align*}
and we set
\begin{align*}
 \PXs := \Big\{ \, \rho \in \PX \ | \ 
    \rho(x) > 0 \quad \forall x \in \cX\ \Big\}\;.
\end{align*}
\begin{definition}\label{def:weight}
  A \emph{weight function} is a function $\theta : [0,\infty) \times
  [0,\infty) \to [0,\infty)$ having the following properties:
  \begin{enumerate}
  \item[(A1)] $\theta$ is continuous on $[0,\infty) \times
  [0,\infty)$ and
    $C^\infty$ on $(0,\infty) \times (0,\infty)$;
  \item[(A2)] $\theta(s,t) = \theta(t,s)$ for $s, t \geq 0$;
  \item[(A3)] $\theta(s,t) > 0$ for $s,t > 0$;
  \item[(A4)] $\theta(r,s) \leq \theta(r,t)$ for all $r \geq 0$ and $0 \leq s \leq t$;
  \item[(A5)] $\theta : [0,\infty) \times [0,\infty) \to [0,\infty)$ is concave.
  \end{enumerate}
\end{definition}
Note that (A5) implies in particular the following estimate:
\begin{align*}
\theta(2 s, 2 t) \leq 2 \theta(s,t)\quad\forall s,t\geq0\ .
\end{align*}

Let us fix a reversible Markov chain $(\cX, Q, \pi)$ and the weight
function $\theta$ throughout the remainder of this section.  We recall
the definition of the associated non-local transportation metric from
\cite{Ma11}, as given in a slightly different (but equivalent) form in
\cite{EM11}.  Actually, in \cite{Ma11,EM11} a slightly less general
case has been considered, which corresponds to $Q(x,x) = -1$ for all
$x \in \cX$. However, it is easily checked that the results extend to
the present setting verbatim.
\begin{definition}[Non-local transportation
  metrics]\label{def:metric}
Fix a weight function $\theta$.
For $\bar\rho_0, \bar\rho_1 \in \PX$ we define 
\begin{align*}
  \cW(\bar\rho_0, \bar\rho_1)^2 := 
  \inf \bigg\{ \int_0^1 \cA(\rho_t, \psi_t) \dd t \ : \ {(\rho, \psi) \in \CE(\bar\rho_0,\bar\rho_1)} \bigg\}\;,
\end{align*}
where $\CE(\bar\rho_0,\bar\rho_1)$ denotes the collection
of pairs $(\rho,\psi)$ satisfying the following conditions:
\begin{align} 
\label{eq:conditions} 
 \left\{ 
 \begin{array}{ll}
{(i)} & \rho : [0,1] \to \R^\cX  \text{ is }C^\infty\;;\\ 
{(ii)} &  \rho_0 = \bar\rho_0\;, \qquad \rho_1 = \bar\rho_1\;; \\
{(iii)} &  \rho_t \in \PX \text{ for all $t \in [0,1]$}\;;\\
{(iv)} & \psi  : [0,1] \to \R^{\cX}  \text{ is measurable}\;;\\ 
{(v)} &  \text{For all $x \in \cX$ and all $t\in (0,1)$ we have}\\
       &\displaystyle{\partial_t \rho_t(x) 
         + \sum_{y \in \cX}
         \big(\psi_t(y) - \psi_t(x)\big)\hrho_t(x,y) Q(x,y) = 0}\;,
\end{array} 
\right.
\end{align}
and for $\rho\in\PX$ and $\psi \in \R^{\cX}$,
\begin{align*}
 \cA(\rho,\psi) 
 := 
   \frac12\sum_{x,y\in \cX} (\psi(y)-\psi(x))^2\hrho(x,y) Q(x,y) \pi(x)\;.
\end{align*}
\end{definition}

The equation appearing in \eqref{eq:conditions} can be regarded as a
``discrete continuity equation''. The analogy with the continuous case
becomes more apparent if we introduce some notation.  For a function
$\psi : \cX \to \R$, we consider the discrete gradient $\nabla \psi :
\cX \times \cX \to \R$ given by
\begin{align*}
 \nabla \psi(x,y) := \psi(y) - \psi(x)\;.
\end{align*}
For a function $\Psi : \cX \times \cX$ we define its
discrete divergence $\nabla \cdot \Psi : \cX \to \R$ by
\begin{align*}
( \nabla \cdot \Psi )(x) 
  := \frac12 \sum_{y \in \cX}  (\Psi(x,y) - \Psi(y,x) ) Q(x,y) \in \R\;.
\end{align*}
With this notation the integration by parts formula
\begin{align*}
 \ip{\nabla \psi, \Psi}_\pi = -\ip{\psi,\nabla\cdot \Psi}_\pi
\end{align*}
holds. Here we write, for $\phi,\psi : \cX \to \R$ and $\Phi,\Psi : \cX \times\cX \to \R$,
\begin{align*}
\ip{\phi, \psi}_\pi &= \sum_{x \in \cX} \phi(x) \psi(x) \pi(x)\;, \\
\ip{\Phi, \Psi}_\pi &= \frac12 \sum_{x,y \in \cX} 
         \Phi(x,y) \Psi(x,y) K(x,y) \pi(x)\;.
\end{align*}
The continuity equation above can now be written as
\begin{align*}
 \partial_t \rho_t + \nabla \cdot (\hrho_t \nabla \psi_t) = 0\;.
\end{align*}
Moreover, for $\rho \in \PX$ we shall write 
\begin{align*}
 \ip{\Phi, \Psi }_\rho
   := 
   \frac12\sum_{\substack{x, y \in \cX\\x\neq y}} \Phi(x,y)\Psi(x,y) \hrho(x,y) Q(x,y) \pi(x)\;,\qquad
   \| \Psi\|_\rho := \sqrt{ \ip{\Psi, \Psi }_\rho}\;.
\end{align*}
With this notation we have 
\begin{align*}
\cA(\rho, \psi):= \| \nabla \psi\|_\rho^2\;.
\end{align*}

Basic properties of $\cW$ have been studied in \cite{Ma11}. Under the
current assumptions we have the following result.

\begin{proposition}[Metric structure]\label{prop:W-metric}
  The function $\cW : \PX \times \PX \to \R$ defines a metric on
  $\PX$. Moreover, $(\PX,\cW)$ is a geodesic space, i.e., for all
  $\bar\rho_0, \bar\rho_1 \in \PX$ there exists a curve $\rho : [0,1]
  \to \PX$ satisfying $\rho_0 = \bar \rho_0$, $\rho_1 = \bar \rho_1$,
  and
\begin{align*}
 \cW(\rho_s, \rho_t) =  |s-t|  \cW(\rho_0, \rho_1) 
\end{align*}
for all $s, t \in [0,1]$.
\end{proposition}

\begin{proof}
  Under somewhat weaker assumptions, it has been proved in
  \cite{Ma11} that $\cW$ defines an extended, i.e., (possibly
  $[0,\infty]$-valued) metric. To show that $\cW$ is finite under the
  current assumptions, we apply \cite[Theorem 3.12]{Ma11}, which
  asserts that it suffices to check that
\begin{align*}
C_\theta := \int_0^1 \frac{1}{\sqrt{\theta(1-r, 1+r)}} \dd  r < \infty \;.
\end{align*}
Using the concavity assumption (A5) we infer that
\begin{align*}
 \theta(1-r, 1+r) \geq \frac{1-r}{2} \theta(2,0) + \frac{1+r}{2} \theta(2,0)
\end{align*}
for $r \in [-1,1]$, which implies that $C_\theta < \infty$.

The final assertion has been proved in \cite[Theorem 3.2]{EM11}.
\end{proof}

Some basic properties of the metric $\cW$ are collected in the following
result, which asserts that the distance $\cW$ is induced by a
Riemannian metric on the interior $\PXs$ of $\PX$.

\begin{proposition}[Riemannian structure]\label{prop:riem-properties}
The restriction of $\cW$ to $\PXs$ is the Riemannian
  distance induced by the following Riemannian structure:
\begin{itemize}
\item[-] the tangent space of $\rho \in \PXs$ can be identified
  with the set of discrete gradients
\begin{align*}
 T_\rho := \{ \nabla \psi \ : \ \psi \in \R^\cX \}
\end{align*}
by means of the following identification: given a smooth curve
$(-\eps, \eps) \ni t \mapsto \rho_t \in \PXs$ with $\rho_0 = \rho$,
there exists a unique element $\nabla \psi_0 \in T_\rho$, such that
the continuity equation \eqref{eq:conditions}(v) holds at $t = 0$.
\item[-] The Riemannian metric on $T_\rho$ is given by the inner product
\begin{align*}
 \ip{\nabla \phi, \nabla \psi}_\rho 
   = \frac12 \sum_{x,y\in \cX} 
 (\phi(x) - \phi(y))(\psi(x) - \psi(y)) \hrho(x,y)Q(x,y) \pi(x)\;.   
\end{align*}
\end{itemize}
\end{proposition}

The following result characterises the geodesic equations for
$\cW$. The equation for $\psi$ is reminiscent of the Hamilton-Jacobi
equation, which describes geodesics in the Wasserstein space.

\begin{proposition}[Geodesic equations]\label{prop:geod}
Let $\bar\rho \in \PXs$ and $\bar\psi \in \R^{\cX}$. On a
sufficiently small time interval around 0, the unique constant speed
geodesic with $\rho_0 = \bar\rho$ and initial tangent vector 
$\nabla \psi_0 = \nabla\bar\psi$ satisfies the following equations:
\begin{equation}\begin{aligned} \label{eq:geod-equs}
\begin{cases}
\partial_t \rho_t(x) + 
   \displaystyle\sum_{y \in \cX}  ( \psi_t(y) - \psi_t(x) ) \hrho_t(x,y) Q(x,y)  = 0 \;,\\
 \partial_t \psi_t(x)  + \displaystyle\frac12
     \displaystyle\sum_{y \in \cX} \big(  \psi_t(x) -\psi_t(y) \big)^2 
     		 \partial_1\theta(\rho_t(x), \rho_t(y))  Q(x,y) = 0 \;.
\end{cases}
\end{aligned}\end{equation}
\end{proposition}

Let $\Delta$ be the discrete Laplacian associated with $Q$, i.e., for
a function $\psi : \cX \to \R$ we set
\begin{align}\label{eq:Laplace}
 \Delta \psi(x) :=  \sum_{y \in \cX} Q(x,y) (\psi(y) - \psi(x))\;.
\end{align}
The operator $\Delta$ is the generator of the continuous time Markov
semigroup associated with $Q$. Note that we have the usual formula
$\Delta = (\nabla \cdot) \nabla$.  By the reversibility assumption,
$\Delta$ is selfadjoint on $L^2(\cX,\pi)$.

 Let $f \in C([0,\infty);\R)$ be a strictly convex function, which is
 smooth on $(0,\infty)$. We consider the associated entropy functional
 $\cF : \PX \to \R$ defined by
\begin{align*}
 \cF(\rho) := \sum_{x \in \cX} f(\rho(x))\pi(x)\;.
\end{align*}

The following result explains the relevance of the non-local
transportation metrics $\cW$.

\begin{proposition}[Gradient flows]\label{prop:grad-flow}
  The heat flow generated by $\Delta$ is the gradient flow of the
  entropy with respect to $\cW$, provided that $\theta$ is given by
\begin{align}\label{eq:gen-log-mean}
 \theta(r,s) := \frac{r -  s}{f'(r) - f'(s)}\;.
\end{align}
\end{proposition}
Let us note that among all weight functions $\theta$
of the form \eqref{eq:gen-log-mean}, the logarithmic mean
\eqref{eq:log-mean-intro} is the only one satisfying $\theta(r,r) =
r$.

\section{Discrete porous medium equations as gradient flows of the entropy}

In this note we shall be concerned with equations of porous
medium-type associated with a reversible Markov chain $(\cX, Q,
\pi)$.
The following assumption will be in force throughout this section.

\begin{assumption}\label{ass:phi}
The function $\phi : [0,\infty) \to \R$ is strictly increasing,
  $C^\infty$ on $(0,\infty)$, and continuous at $0$.
\end{assumption}

We shall study the \emph{discrete porous medium equation}
\begin{align}\label{eq:pm}
 \partial_t \rho_t = \Delta \phi(\rho_t)\;,
\end{align}
where $\Delta$ denotes the discrete Laplacian defined in
\eqref{eq:Laplace}.  Of course, if $\phi(r) = r$, we recover the usual
discrete heat equation associated with $Q$, which has been studied in
\cite{CHLZ11,Ma11,Mie11a}.  We shall analyse these equations using gradient flow
methods, first in a Hilbertian setting of discrete Sobolev spaces, and
then in a Riemannian setting of non-local transportation metrics.

\subsection{A Hilbertian gradient flow structure}

In order to apply Hilbertian gradient flow methods we shall introduce
a ``discrete $H^{-1}$-space'' $\sH^{-1}$.  First we observe that, by
the irreducibility assumption, $\Ker(\Delta) = \lin\{\one\}$, where
$\one$ denotes the function identically equal to $1$. Since $\Delta$
is selfadjoint on $L^2(\cX, \pi)$, it follows that the operator
$\Delta$ is bijective on
\begin{align*}
 \Ran(\Delta) := \Big\{ \psi : \cX \to \R \ | \ \sum_{x \in \cX} \psi(x) \pi(x) = 0 \Big\}\;.
\end{align*}
For $\psi, \psi_1, \psi_2 \in \Ran(\Delta)$ it thus makes sense to
define the inner product
\begin{align*}
 \ip{\psi_1, \psi_2}_{\sH^{-1}} := 
 \ip{-\Delta^{-1} \psi_1, \psi_2}_\pi\;,
\end{align*}
where $\ip{\cdot,\cdot}_\pi$ denotes the $L^2(\cX,\pi)$-inner product. The associated norm is given by
\begin{align*}
 \| \psi \|_{\sH^{-1}} 
   :=\sqrt{\ip{\psi, \psi}_{\sH^{-1}}}
    = \| \nabla \Delta^{-1} \psi \|_{\pi}
    = \sqrt{
    \frac12\sum_{x,y\in \cX} \big(\Delta^{-1}\psi(y)-\Delta^{-1}\psi(x)\big)^2Q(x,y) \pi(x)}\;.
\end{align*}
For arbitrary functions $\psi : \cX \to \R$ we set $c_\psi = \sum_{x
  \in \cX} \psi(x)\pi(x)$ and note that $\psi - c_\psi\one$ belongs to
$\Ran(\Delta)$. A Hilbertian norm on $\R^\cX$ can then be defined by
\begin{align*}
 \| \psi \|_{\sH^{-1}}
    := \sqrt{c_\psi^2  + \| \psi - c_\psi\one\|_{\sH^{-1}}^2}\;.
\end{align*}

The equation $\partial_t \rho_t = \Delta \phi(\rho_t)$ can now 
be studied using classical Hilbertian gradient flow arguments:

\begin{proposition}\label{prop:classical}
Let $\Phi(r) = \int_0^r \phi(s) \dd s$ and consider the functional 
$\PPhi : \R^{\cX} \to \R \cup \{+ \infty\}$ defined by
\begin{align*}
 \PPhi(\rho) := \sum_{x \in \cX} \Phi(\rho(x)) \pi(x)
\end{align*}
for $\rho \in \PX$, and $\PPhi(\rho) = + \infty$ otherwise.  Then the
functional $\PPhi$ is strictly convex on $\PX$ and attains its unique
minimum at $\one$.  As a consequence, the following assertions hold
for all $\bar\rho_0 \in \PX$:
\begin{enumerate}
\item Among all locally absolutely continuous curve $(\rho_t)_t$ in
  $\PX$ with $\rho_0 = \bar \rho_0$, there exists a unique one
  satisfying the evolution variational inequality 
\begin{align}\label{eq:hilbert-evi}
 \frac12 \ddt\| \rho_t - \sigma \|_{\sH^{-1}}^2 
    \leq \PPhi(\sigma) - \PPhi(\rho_t)
\end{align}
for all $\sigma \in \PX$ and a.e. $t \geq 0$.
\item For all $\bar\rho_0 \in
  \PX$ and a.e. $t\geq 0$ the curve $(\rho_t)$ satisfies the ordinary
  differential equation
\begin{align*}
 \partial_t \rho_t = \Delta \phi(\rho_t)\;.
\end{align*}
\item For all $\bar\rho_0 \in \PX$ we have $\rho_t \to \one$ as $t \to
  \infty$.
\end{enumerate}
\end{proposition}

\begin{proof}
  Strict convexity of $\PPhi$ follows from the assumption that $\phi$
  is strictly increasing. Part (1) follows then from the theory of
  gradient flows of convex functionals in Hilbert spaces
  \cite[Th\'eor\`eme 3.1]{Bre73} or the more general theory in metric
  spaces \cite[Theorem 4.0.4]{AGS08}.  To prove (2), it suffices to
  compute the $\sH^{-1}$-gradient of $\PPhi$. To do this, note that
  for a smooth curve $(\rho_t)$ we have
\begin{align*}
 \ddt \PPhi(\rho_t) 
   =  \ddt \ip{\Phi(\rho_t), \one}_\pi
   =   \ip{\phi(\rho_t), \partial_t\rho_t}_\pi
   =   -\ip{\Delta\phi(\rho_t), \partial_t\rho_t}_{\sH^{-1}}\;,
\end{align*}
which shows that the $\sH^{-1}$-gradient of the functional $\PPhi$ at
$\rho \in\PX$ is given by $-\Delta\phi(\rho)$, as desired.

Since the gradient of the convex functional $\PPhi$ vanishes at
$\one$, it follows that $\PPhi$ attains it minimum at
$\one$. (Alternatively, this follows from Jensen's inequality.)
Uniqueness of the minimiser follows from the strict convexity of
$\PPhi$. Part (3) is then a consequence of \cite[Corollary
4.0.6]{AGS08}.
\end{proof}

\subsection{A gradient flow structure for non-local transportation metrics}

The gradient flow structure described in Proposition
\ref{prop:classical} can be regarded as a discrete analogue of the
$H^{-1}$-gradient flow structure of the porous medium equation.  Let
us now focus on different gradient flow structures for the same
equation, which are analogous to the Wasserstein-gradient flow
structure. We refer to \cite[Section 2]{O01} for a discussion of
different gradient flow structures for the porous medium equation in
the continuous setting.

\begin{assumption}\label{ass:f}
The function $f : [0,\infty) \to \R$ is strictly convex,
  $C^\infty$ on $(0,\infty)$, and continuous at $0$.
Moreover, the function $\theta_{\phi,f}$ defined by
\begin{align*}
 \theta_{\phi,f}(r,s) := \frac{\phi(r) - \phi(s)}{f'(r) - f'(s)}
\end{align*}
for $r \neq s$, extends to a weight function on $[0,\infty) \times [0,\infty)$ in the sense of Definition \ref{def:weight}.
\end{assumption}

The associated metric on $\PX$ will be denoted by 
$\cW_{\phi, f}$. Often, if confusion is unlikely to arise, we will  simply write $\cW$.

We consider the \emph{entropy functional} $\cF : \PX \to \R$ defined by
\begin{align}\label{eq:f-entropy}
 \cF(\rho) := \sum_{x \in \cX} f(\rho(x)) \pi(x)\;.
\end{align}
Of course, if $f(r) = r \log r$ we recover the usual relative entropy
with respect to $\pi$.

\begin{proposition}\label{prop:grad-F}
  Let $\cW$ be any non-local transportation metric. The $\cW$-gradient
  of the functional $\cF$ at $\rho \in\PXs$ is given by $\nabla
  f'(\rho)$.
\end{proposition}

Here we use the identification of the tangent space provided in
Proposition \ref{prop:riem-properties}.

\begin{proof}
  Let $\rho \in \PXs$ and pick a smooth curve $(\rho_t)$ satisfying
  the continuity equation
\begin{align*}
\partial_t \rho_t + \nabla \cdot (\hrho_t \nabla \psi_t) = 0\;,
\end{align*}
where we use the notation from Section \ref{sec:prelim}.
It then follows that
\begin{align*}
 \ddt \cF(\rho_t) 
   =  \ddt \ip{f(\rho_t), \one}_\pi
   =  -\ip{f'(\rho_t), \nabla \cdot (\hrho_t \nabla \psi_t) }_\pi
   =   \ip{\nabla f'(\rho_t), \nabla \psi_t}_{\rho_t}\;,
\end{align*}
hence the $\cW_{\phi, f}$-gradient of the functional $\cF$ at $\rho
\in\PXs$ is given by $\nabla f'(\rho)$.
\end{proof}

The next result is now an immediate consequence.

\begin{theorem}\label{thm:gradient-flow}
  The gradient flow equation 
  of the entropy functional $\cF$ with respect to the
  metric $\cW_{\phi,f}$ is given by the
  porous medium equation $\partial_t \rho_t = \Delta \phi(\rho_t)$.
\end{theorem}

\begin{proof}
  The gradient flow equation of the smooth functional $\cF$ on $\PXs$
  is given by
\begin{align*}
 \partial_t \rho_t + \nabla \cdot (\hrho_t \nabla \psi_t) = 0\;,\qquad 
 \nabla \psi_t = - \grad_\cW \cF(\rho_t)\;.
\end{align*}
Using Proposition \ref{prop:grad-F} and the fact that $\hrho\, \nabla
f'(\rho) = \nabla \phi(\rho)$ by the definition of $\theta_{\phi,f}$,
we infer that
\begin{align*}
 \nabla \cdot (\hrho_t \nabla \psi_t)
   = 
 - \nabla \cdot (\hrho_t \nabla f'(\rho_t))
   = - \Delta \phi(\rho_t)\;,
\end{align*}
hence the gradient flow equation reduces to the porous medium equation
$\partial_t \rho_t =\Delta \phi(\rho_t)$, as desired.
\end{proof}

Clearly, the gradient flow structure for the heat equation given in
Proposition \ref{prop:grad-flow} corresponds to the special case of
Theorem \ref{thm:gradient-flow} where $\phi(r) = r$.

Let us note that also the $\sH^{-1}$-gradient flow structure from
Proposition \ref{prop:classical} is a special case of Theorem
\ref{thm:gradient-flow}. Indeed, if $f = \Phi$, then $\cF =
\PPhi$. Moreover, the following result asserts that the
$\sH^{-1}$-distance on $\PX$ coincides with the non-local
transportation metric $\cW_\one$, which is given by Definition
\ref{def:metric} in the special case where $\theta(r,s) = 1$ for all
$r,s \geq 0$.

\begin{proposition}\label{prop:Hminus1}
For $\rho_0, \rho_1 \in \PX$ we have
\begin{align*}
 \| \rho_0 - \rho_1 \|_{\sH^{-1}}
   = \cW_\one(\rho_0, \rho_1)\;, 
\end{align*}
where $\cW_\one$ denotes the discrete transportation metric with
$\theta \equiv 1$.
\end{proposition}

\begin{proof}
  Note that the geodesic equations \eqref{eq:geod-equs} for $\cW_\one$
  reduce to
\begin{align*}
 \partial_t \rho_t = - \Delta \psi_t\;, \qquad
 \partial_t \psi_t = 0\;.
\end{align*}
Since $\psi$ does not depend on $t$, we infer that the
$\cW_\one$-geodesic between $\rho_0$ and $\rho_1$ in $\PX$ is given by
linear interpolation, i.e., $\rho_t = (1-t) \rho_0 + \rho_1$.  It
follows that $\Delta\psi_t = \rho_0 - \rho_1$ for all $t \in [0,1]$,
and since $\| \cdot \|_{\rho_t} = \| \cdot\|_\one$ we obtain
\begin{align*}
 \cW_\one(\rho_0, \rho_1)^2 
    = \int_0^1 \| \nabla \psi_t \|_{\rho_t}^2 \dd t
    = \| \nabla \Delta^{-1}(\rho_0 - \rho_1) \|_{\pi}^2
    =  \| \rho_0 - \rho_1 \|_{\sH^{-1}}^2\;.
\end{align*}
\end{proof}

\subsection{Porous medium equations and R\'enyi entropy}
\label{sec:Potenzmittel}
Let us now specialise to a particularly interesting setting, motivated
by the Wasserstein gradient flow structure for porous medium equations
of the form $\partial_t \rho_t= \Delta \rho_t^m$ in $\R^n$ from
\cite{O01}.

Let $0 < m \leq 2$ and $m \neq 1$, and consider the functions
\begin{align*}
 \phi(r) = \phi_m(r) = r^m\;,\qquad  f(r) = f_m(r) = \frac{1}{m-1} r^m\;.
\end{align*}
In this case the porous medium equation and the entropy functional are given by
\begin{align*}
\partial_t \rho_t = \Delta \rho_t^m\;,\qquad
 \cF(\rho) =  \cF_m(\rho) := \frac{1}{m-1}\sum_{x \in \cX} \rho(x)^m \pi(x)\;,
\end{align*}
thus $\cF_m$ is the usual R\'enyi entropy. The weight function is given by
\begin{align}\label{eq:thetam}
 \theta_m(r,s) := \theta_{\phi_m,f_m}(r,s) 
			  := \frac{m-1}{m}\frac{r^m - s^m}{r^{m-1} - s^{m-1}}\;,
\end{align}
and we let $\cW_m$ denote the associated non-local transportation metric. 
Note that 
\begin{align*}
 \lim_{m \to 1} \theta_m(r,s)  = \theta_1(r,s)\;,
\end{align*}
where $\theta_1$ denotes the logarithmic mean defined in
\eqref{eq:log-mean-intro}. The weight function $\theta_m$ admits the remarkable integral representation
\begin{align}\label{eq:integral-rep}
 \theta_m(r,s)= 
   \int_0^1 \Big( (1-\alpha) r^{m-1} + \alpha s^{m-1} \Big)^{\frac{1}{m-1}} \dd \alpha\;,
\end{align}
which can be checked by an explicit calculation and naturally generalises \eqref{eq:log-mean-intro}.
We collect here some basic properties of the weight functions
associated to the R\'enyi entropy.

\begin{lemma}\label{lem:theta-m}
 Let $0 < m \leq 2$. Then the function $\theta_m$ defines a 
  weight function with the following additional properties:
 \begin{enumerate}
 \item (Homogeneity) $\theta_m(\lambda s,\lambda t)=\lambda
   \theta_m(s,t)$ for $\lambda>0$ and $s,t>0$.
 \item (Monotonicity in $m$) For $m<m'$ we have $\theta_m(s,t) \leq
   \theta_{m'}(s,t)$ for all $s,t>0$.
 \item (Zero at the boundary) $\theta_m(0,t)=0$ for all
   $t\geq0$ if and only if $m\leq 1$.
 \end{enumerate}
\end{lemma}

\begin{proof}
  Properties (A1) -- (A4) in Definition \ref{def:weight} are obvious
  from either \eqref{eq:thetam} or \eqref{eq:integral-rep}.

  Concavity of $\theta_m$ follows from \eqref{eq:integral-rep}, since
  an explicit computation shows that the Hessian of the function
\begin{align*}
F : (r,s) \mapsto
   \int_0^1 \Big( (1-\alpha) r^{m-1} + \alpha s^{m-1} \Big)^{\frac{1}{m-1}}
\end{align*}
is given by
\begin{align*}
 \Hess F(r,s) =(1-\alpha) \alpha (2-m) 
     \Big( (1-\alpha) r^{m-1} + \alpha s^{m-1} \Big)^{\frac{1}{m-1}-2}
     r^{m-2} s^{m-2}
\left(\begin{array}{cc}-s^2 & rs \\rs & -r^2\end{array}\right)\;,
\end{align*}
which is negative semi-definite for all $\alpha \in [0,1]$.

Property (2) follows from the monotonicity in $p$ of $L^p$-``norms'',
applied on a two-point space $\{0,1\}$ with probability measure
$\mu_\alpha := (1-\alpha)\delta_0 + \alpha \delta_1$. This
monotonicity follows from Jensen's inequality and holds for negative
$p$ as well.

Finally, (3) follows from an explicit computation and (1) is obvious.
\end{proof}

Let us note that $\theta_m$ is not a weight function for $m > 2$,
since in this case Minkowski's inequality for $L^{m-1}$-norms implies
that $\theta_m$ is convex.

\begin{corollary}\label{cor:finite-distance}
  For any reversible Markov chain $(\cX, Q, \pi)$ 
  the distance $\cW_m$ is finite and non-increasing in
  $m$. More precisely, for any $0 < m \leq m'\leq 2$ and $\rho_0,\rho_1\in\PX$ we
  have 
  \begin{align*}
    \cW_{m'}(\rho_0,\rho_1) \leq \cW_{m}(\rho_0,\rho_1) < \infty\ .
  \end{align*}
\end{corollary}

\begin{proof}
  This is an immediate consequence of the monotonicity of $\theta_m$ stated in
  Lemma \ref{lem:theta-m} combined with the equivalent definition of the
  distance $\cW$ given in \cite[Lemma 2.9]{EM11}.
\end{proof}

Applied to $\phi = \phi_m$ and $f = f_m$,
Theorem \ref{thm:gradient-flow} reduces to the following result.

\begin{corollary}\label{thm:gradient-flow-potenz}
  Let $0 < m \leq 2$. Then the gradient flow equation of the R\'enyi
  entropy $\cF_m$ with respect to the metric $\cW_m$ is the porous
  medium equation $\partial_t \rho_t = \Delta \rho_t^m$.
\end{corollary}

\section{Geodesic \texorpdfstring{$\kappa$}{k}-convexity of entropy
  functionals}\label{sec:convexity}

In this section we analyse convexity properties of entropy functionals
along geodesics for non-local transportation metrics. We fix a
reversible Markov chain $(\cX, Q, \pi)$, and we fix $\phi, f :
[0,\infty) \to \R$ satisfying Assumptions \ref{ass:phi} and
\ref{ass:f}. Throughout this section we set
\begin{align}\label{eq:theta-phi-f}
 \theta(r,s) = \theta_{\phi, f}(r,s) = \frac{\phi(r) - \phi(s)}{f'(r) - f'(s)}\;.
\end{align}
Let $\cW$ be the associated metric, and let $\cF$ be the entropy
functional defined in \eqref{eq:f-entropy}.

\subsection{The Hessian of entropy functionals}

We start by calculating the Hessian of an entropy functional $\cF$ in
the Riemannian structure induced by $\cW$.  For a probability density
$\rho \in \PXs$ we introduce the function $\hDelta_\phi \rho : \cX
\times \cX \to \R$ defined by
\begin{align*}
\hDelta_\phi \rho(x,y) :=  
 \partial_1 \theta(\rho(x), \rho(y)) \Delta \phi(\rho)(x)
  + \partial_2 \theta(\rho(x), \rho(y)) \Delta \phi(\rho)(y)\;.
\end{align*}
Furthermore, for a function $\psi :  \cX \to \R$ we set
\begin{equation}\begin{aligned}\label{eq:B}
 \cB(\rho, \psi) 
&= \ \frac12
  \bip{\hDelta_\phi\rho\, \cdot\, \nabla \psi, \nabla \psi }_\pi
    - \bip{\hrho  \cdot\,\nabla \psi\ ,\, \nabla (\phi'(\rho)\Delta\psi)}_\pi 
\\ & = \ 
  \frac14
       \sum_{x,y,z \in \cX} 
         \big(\psi(x) - \psi(y)\big)^2 Q(x,y) \pi(x)
         \\ & \qquad \times
        \Big( \partial_1 \theta\big(\rho(x), \rho(y)\big) \big(\phi(\rho(z)) - \phi(\rho(x)) \big)Q(x,z)
 \\& \qquad\quad
    + \partial_2 \theta\big(\rho(x), \rho(y)\big) \big(\phi(\rho(z)) - \phi(\rho(y)) \big)Q(y,z)\Big)
  \\ &  - 
   \frac12
       \sum_{x,y,z \in \cX}
          \big(\psi(x) - \psi(y)\big)
        \hrho(x,y)Q(x,y)\pi(x)
  \\& \qquad \times  
   \Big( \phi'(\rho(x)) Q(x,z) \big( \psi(z) - \psi(x) \big)
           -   \phi'(\rho(y)) Q(y,z) \big( \psi(z) -\psi(y)  \big) \Big)  \;.
\end{aligned}\end{equation}
The proof of the following result is an adaptation of \cite[Proposition 4.2]{EM11}.

\begin{proposition}\label{prop:Hessian}
For $\rho \in \PXs$ and $\psi \in \R^\cX$ we have 
\begin{align*}
\bip{ \Hess \cF(\rho) \nabla \psi\ ,\, \nabla \psi}_\rho
    = \cB(\rho,\psi)\;.
\end{align*}

\begin{proof}
  Take $(\rho_t , \psi_t)_t$ satisfying the geodesic equations
  \eqref{eq:geod-equs}.
Using the continuity equation and \eqref{eq:theta-phi-f} we obtain
\begin{align*}
 \ddt \cF(\rho_t)
    &=  - \bip{ f'(\rho_t), \nabla\cdot( \hat\rho_t \nabla \psi_t) }_\pi  \\&  =  \bip{ \nabla f'(\rho_t) ,  \hat\rho_t \cdot \nabla \psi_t }_\pi
  \\&  =  \bip{ \nabla \phi(\rho_t) ,  \nabla \psi_t }_\pi\;.
\end{align*}
Therefore the second derivative of the entropy is given by
\begin{align*}
  \ddtt \cF(\rho_t)
    & =  \bip{ \nabla \partial_t \phi(\rho_t) ,  \nabla \psi_t }_\pi
        + \bip{ \nabla \phi(\rho_t) ,  \nabla \partial_t \psi_t }_\pi 
   \\&  = -  \bip{  \partial_t \phi(\rho_t) ,  \Delta \psi_t }_\pi
         -  \bip{ \Delta  \phi(\rho_t) ,  \partial_t \psi_t }_\pi \ .
\end{align*}
Using the continuity equation we obtain
\begin{align*}
  \bip{  \partial_t \phi(\rho_t) ,  \Delta \psi_t }_\pi
   & =   \bip{   \phi'(\rho_t) \partial_t \rho_t,  \Delta \psi_t }_\pi
 \\& = - \bip{   \nabla\cdot( \hat\rho \nabla \psi_t) , \phi'(\rho_t) \Delta \psi_t }_\pi
 \\& =   \bip{ \hat\rho_t \nabla \psi_t ,  \nabla(\phi'(\rho_t) \Delta \psi_t ) }_\pi\;.
\end{align*}
Furthermore, applying the geodesic equations \eqref{eq:geod-equs} and
the detailed balance equations \eqref{eq:detailed-balance}, we infer
that
\begin{align*}
 \bip{ \Delta \phi(\rho_t) ,  \partial_t \psi_t }
 & 
  =
   -\frac12
       \sum_{x,y,z \in \cX}
         \big(\psi_t(x) - \psi_t(y)\big)^2 \partial_1 \theta\big(\rho_t(x), \rho_t(y)\big) 
\\& \qquad\qquad\qquad  \times         \big(\phi(\rho_t(z))  - \phi(\rho_t(x))\big) Q(x,y) Q(x,z) \pi(x)
 \\&  =
   -\frac14
       \sum_{x,y,z \in \cX}
         \big(\psi_t(x) - \psi_t(y)\big)^2 
        \Big( \partial_1 \theta\big(\rho_t(x), \rho_t(y)\big) \big(\phi(\rho_t(z)) - \phi(\rho_t(x)) \big)Q(x,z)
 \\& \qquad\qquad\qquad
   + \partial_2 \theta\big(\rho_t(x), \rho_t(y)\big) \big(\phi(\rho_t(z)) - \phi(\rho_t(y)) \big)Q(y,z)\Big)
       Q(x,y)  \pi(x)
 \\& = -\frac12 \bip{\hDelta_\phi\rho_t\, \cdot\, \nabla \psi_t, \nabla \psi_t }_\pi\;.
\end{align*}
Combining the latter three identities, we obtain
\begin{align*}
\bip{ \Hess \cF(\rho_t) \nabla \psi_t\ ,\, \nabla \psi_t}_{\rho_t}
  &= \ddtt \cF(\rho_t)
  \\&   = - \bip{  \hat\rho_t \cdot   \nabla \psi_t ,  \nabla\Delta \psi_t }_{\pi}
        +  \frac12 \bip{\hDelta_\phi\rho_t\, \cdot\, \nabla \psi_t, \nabla \psi_t }_\pi\;,
\end{align*}
which is the desired identity.
\end{proof}
\end{proposition}

Since the Riemannian metric does not necessarily extend continuously to the
boundary of $\PX$, it is not obvious that a lower bound $\kappa$ on
the Hessian in the interior $\PXs$ implies geodesic $\kappa$-convexity
in the metric space $(\PX,\cW)$. Nevertheless, an Eulerian argument by
Daneri and Savar\'e \cite{DS08} can be adapted to the current setting,
to show that this is indeed the case. We refer to \cite[Theorem
4.5]{EM11} and \cite[Proposition 2.1]{Mie11b} for the details in the case where
$\theta$ is the logarithmic mean.

\begin{proposition}\label{prop:Daneri-Savare}
  For $\kappa \in \R$ the following assertions are equivalent:
\begin{enumerate}
\item For every constant speed geodesic $(\rho_t)_{t \in [0,1]}$
  in $(\PX, \cW)$ we have
\begin{align}\label{eq:K-convex}
  \cF(\rho_t) \leq
  (1-t) \cF(\rho_0) + t \cF(\rho_1) - \frac\kappa{2} t(1-t) \cW(\rho_0, \rho_1)^2\;.
\end{align}
\item For all $\bar\rho_0 \in \PX$, the solution $(\rho_t)$ to the
  porous medium equation with $\rho_0 = \bar\rho_0$ from Proposition
  \ref{prop:classical}, satisfies the evolution variational inequality
\begin{align}\label{eq:EVI}
  \frac{1}{2}\ddt \cW(\rho_t, \sigma)^2 + \frac{\kappa}{2}
\cW(\rho_t, \sigma)^2 \leq \cF(\sigma) - \cF(\rho_t)
\end{align}
for all $\sigma \in \PX$ and a.e. $t \geq 0$.
\item For all $\rho \in \PXs$ we have 
\begin{align*}
\Hess \cF(\rho) \geq \kappa\;.
\end{align*}
\item For all $\rho \in \PXs$ and $\psi \in \R^\cX$ we have
\begin{align*}
 \cB(\rho, \psi) \geq \kappa \cA(\rho, \psi)\;.
\end{align*}
\end{enumerate}
\end{proposition}

The evolution variational inequality \eqref{eq:EVI} can be regarded as
a definition of a (strong) notion of a gradient flow in the setting of
a metric space.  This inequality has been extensively studied
recently, see for example \cite{AGS08} and \cite[Section 3]{DS08}.

\subsection{Examples}
In this section we will study $\cW$-geodesic convexity of $\cF$ in
some simple concrete examples.  To simplify notation let us write, for
$i = 1, 2$,
\begin{align*}
 \partial_i \hrho(x,y) :=  \partial_i \theta(\rho(x),\rho(x))\;.
\end{align*}
Set
\begin{align}\label{eq:kappa-q}
\kappa_Q := \sup \{ \kappa \in \R \ : \ \cF \text{ is $\kappa$-convex along $\cW_{\phi,f}$-geodesics}  \}\;,
\end{align}
with the understanding that $\sup\{\emptyset\} = - \infty$.

\subsubsection{The two-point space}
We consider the two-point space $\cX = \{a, b\}$ endowed with the
$Q$-matrix $Q$ defined by
\begin{align*}
Q(a,b) = p\;,
\qquad  Q(b,a) = q\;,
\end{align*}
for $p, q > 0$. In this case the stationary probability measure $\pi$
is given by
\begin{align*}
 \pi(a)  = \frac{q}{p+q}\;, \qquad \pi(b)  = \frac{p}{p+q}\;.
\end{align*}
In this case we have the following characterisation of $\kappa_Q$ in
terms of $p$ and $q$.

\begin{proposition}[Geodesic convexity on the two-point
  space]\label{prop:2pt-convexity}
  Let $Q$ be as above. Then
\begin{align*}
 \kappa_Q =  \frac12 \inf \bigg\{
 p \phi'(r) + q \phi'(s) + \theta(r,s) \big(p f''(r) + q f''(s)\big)
 \bigg\}\;,
\end{align*}
where $r = \frac{p+q}{2q}(1-\alpha)$, $s = \frac{p+q}{2p}(1+\alpha)$,
and the infimum runs over all $\alpha \in (-1,1)$.
\end{proposition}

\begin{proof}
An explicit computation shows that
\begin{align*}
 \cA(\rho, \psi)
    = \frac{pq}{p+q} 
     \frac{\phi(\rho(a)) - \phi(\rho(b))}{f'(\rho(a)) - f'(\rho(b))}
     (\psi(a) - \psi(b))^2\;,
\end{align*}
and
\begin{align*}
\cB(\rho, \psi)
&  =\frac{pq}{p+q} \bigg[  \frac12  
    \Big( q \partial_2 \hrho(a,b) - p \partial_1 \hrho(a,b) \Big)
    \Big( \phi(\rho(a)) - \phi(\rho(b)) \Big)
\\&  \qquad\qquad+
   \hrho(a,b)
    \Big( p \phi'(\rho(a)) + q \phi'(\rho(b)) \Big)\bigg]
    ( \psi(a) - \psi(b))^2\;.
\end{align*}
In view of Proposition \ref{prop:Daneri-Savare}, the result follows
from these expressions.
\end{proof}

The following result provides a simplified expression for $\kappa_Q$
in the case where $\pi$ is symmetric and $\theta = \theta_m$.  The
case $m=1$ has already been considered in \cite[Proposition
2.12]{Ma11}.

\begin{corollary}
If $p = q$ and $\phi = \phi_m$, and $f = f_m$, then
\begin{align*}
  \kappa_Q &= 
   \left\{ \begin{array}{ll}
\displaystyle{\frac{pm}2} \inf_{\alpha \in (-1,1)} \bigg\{
\displaystyle{\frac{m-1}{m}
\frac{   (1-\alpha)^{m} - (1+\alpha)^{m}}
	 {(1-\alpha)^{m-1} - (1+\alpha)^{m-1}}}
  \Big( (1-\alpha)^{m-2} + (1+\alpha)^{m-2}\Big)
\\\qquad\qquad\qquad
 + (1-\alpha)^{m-1} + (1+\alpha)^{m-1}  \bigg\}\;,
   & \hspace{-13ex}0 < m \leq 2\;, \  m \neq 1\;,\\
p \inf_{\alpha \in (-1,1)} \bigg\{
\displaystyle{\frac{1}{1-\alpha^2}\frac{\alpha}{\arctanh(\alpha)}
+1 } \bigg\}
 = 2p\;, 
   & m = 1\;.\end{array} \right.\\
\end{align*}
\end{corollary}

\begin{proof}
This follows readily from Proposition \ref{prop:2pt-convexity}.
\end{proof}

\subsubsection{The discrete circle}

As announced in Proposition \ref{prop:intro-counterexample}, we will
exhibit an instance of the discrete porous medium equation where
convexity fails.  For this purpose, we consider the discrete circle of
length $N$, i.e., $\cX = \cT_N = \Z / N \Z$ for some $N \geq 2$. All
computations below are understood modulo $N$. Let $Q$ denote the
discrete Laplacian, normalised so that $Q(x,y) = q$ if $|x - y| = 1$
and $Q(x,y) = 0$ otherwise. In this case $\pi$ is the uniform
probability given by $\pi(x) = \frac1N$ for all $x \in \cX$.

\begin{proposition}\label{prop:counterexample}
  Let $m = 2$, let $f_m$, $\phi_m$ and $\theta_m$ be as in Section
  \ref{sec:Potenzmittel}, and let $\cW$ be the associated non-local
  transportation metric.  Let $Q$ be the discrete Laplacian on the
  discrete circle $\cT_N=\Z/N\Z$ with $N\geq6$ defined above.  Then
  the functional $\cF_m$ is not convex along $\cW$-geodesics. More
  precisely,
\begin{align*}
 \kappa_Q \leq -\frac{qN}{2}\;.
\end{align*}
\end{proposition}

\begin{proof}
Since $m =2$ we have the simple expressions
\begin{align}\label{eq:m-2}
\phi(r) = f(r) = r^{2}\;,\quad
\theta(r,s) = \frac{r+s}{2}\;,\quad
\partial_1 \theta(r,s) =\partial_2\theta(r,s) = \frac12\;.
\end{align}
Inserting these into \eqref{eq:B}, we calculate the Hessian of the
R\'enyi entropy as

\begin{align*}
  \cB(\rho,\psi) & = \frac{q^2}{4N} \sum_{i=1}^N (\psi_i - \psi_{i+1})^2
  \bigg(\rho_{i-1}^2+\rho_{i+2}^2+3\rho_{i}^2+3\rho_{i+1}^2+8\rho_{i}\rho_{i+1} \bigg) 
\\& \quad - \frac{q^2}{N} \sum_{i=1}^N (\psi_i -
  \psi_{i+1})(\psi_{i+1} - \psi_{i+2}) \bigg( 2\rho_{i+1}^2 + \rho_{i+1}\big(\rho_{i}+\rho_{i+2}\big)\bigg)\;.
\end{align*}
Choosing  in particular 
\begin{align*}
  \psi=(\psi_1,\dots,\psi_N)=(0,1,2,\dots,2,0)\;,\quad\rho=(\rho_1,\dots,\rho_N)=(\eps,N-(N-1)\eps,\eps,\dots,\eps)
\end{align*}
we see that only three terms in the first and one term in the second
sum are non-zero, and we find 
\begin{align*}
  \cB(\rho,\psi) & = -\frac{q^2N}{2} + O(\eps) \;.
\end{align*}
Since $\cA(\rho,\psi) = q + O(\eps)$, the result follows.
\end{proof}

Let us remark that in the case where $q = N^2$, the discrete Laplacian
converges to the continuous Laplacian on the torus $\mathbb{T} = \R /
\Z$. In this limit space, it is well known that the R\'enyi entropy
$\cF_2$ is convex along $2$-Wasserstein geodesics in
$\cP(\mathbb{T})$. However, the previous result shows that the
behaviour at the discrete level is very different, since $\kappa_Q
\leq - N^3/2$.

\subsection{Consequences of geodesic convexity}

In this subsection we will show that convexity of the entropy
functional along $\cW$-geodesics implies a contraction property for
solutions to the porous medium equation as well as a number of
functional inequalities for the invariant measure of the Markov
chain. These results are in the spirit of the work of Otto--Villani
\cite{OV00}. Similar results for the heat flow associated to a finite
Markov chain have been obtained in \cite{EM11}.

As a first result we single out the following $\kappa$-contractivity
property, which is a direct consequence of the evolution variational
inequality \eqref{eq:EVI}.

\begin{proposition}[$\kappa$-Contractivity of the PME]\label{prop:contractivity}
  Suppose that $\kappa_Q \in \R$ and let $(\rho_t)_{t},
  (\sigma_t)_{t}$ be two solutions of the discrete porous medium
  equation as given by Proposition \ref{prop:classical}.  Then we have
  for all $t\geq0$:
\begin{align*}
 \cW(\rho_t, \sigma_t) \leq e^{-\kappa t} \cW(\rho_0, \sigma_0)\ .
\end{align*}
\end{proposition}

\begin{proof}
  This follows from Proposition \ref{prop:Daneri-Savare} by applying
  \cite[Proposition 3.1]{DS08} to the functional $\cF$ on the metric
  space $(\PX,\cW)$.
\end{proof}

We introduce the following functional, which we regard as an analogue
of the Fisher information:
\begin{align*}
  \cI(\rho) := \frac12\sum_{x,y\in \cX}\big[f'(\rho(y))-f'(\rho(x))\big]\big[\phi(\rho(y))-\phi(\rho(x))\big]Q(x,y)\pi(x)\ .
\end{align*}
If $f$ is not differentiable at $0$, we use the convention that
$\cI(\rho)=+\infty$ for $\rho\notin\cP_*(\cX)$. Using Proposition
\ref{prop:grad-F}, we see that for $\rho\in\cP_*(\cX)$ we have
\begin{align*}
  \cI(\rho) = \| \grad_\cW\cF(\rho) \|_\rho\ .
\end{align*}
The significance of this functional is due to the fact that it gives
the change, or dissipation, of the entropy functional along solutions
of the equation $\partial_t\rho_t=\Delta\phi(\rho_t)$, namely:
\begin{align*}
  \ddt\cF(\rho_t) = -\cI(\rho_t)\ .
\end{align*}
Note that in the setting of Section \ref{sec:Potenzmittel}, the
dissipation functional $\cI_m$ associated to $\phi_m$ and $f_m$ is
given by
\begin{align*}
  \cI_m(\rho) := \frac12\frac{m}{m-1}\sum\limits_{x, y\in \cX}\big[\rho(y)^{m-1}-\rho(x)^{m-1}\big]\big[\rho(y)^m-\rho(x)^m\big]Q(x,y)\pi(x)\ .
\end{align*}

As before, let $\one \in \PX$ denotes the density of the stationary
distribution, which is everywhere equal to $1$. It follows from the
definition that $\cF(\one) = f(1)$.

We introduce the following functional inequalities.

\begin{definition}\label{def:functional-inequalities}
  The Markov chain $(\cX, Q, \pi)$ satisfies 
\begin{enumerate}
\item an $\cF\cW\cI$ inequality with
  constant $\kappa\in\R$ if for all
  $\rho\in\PX$,
  \begin{align}\label{eq:FWI}
  \tag{{$\cF\cW\cI$}($\kappa$)}
    \cF(\rho) - \cF(\one) \leq \cW(\rho,\one)\sqrt{\cI(\rho)}-\frac{\kappa}{2}\cW(\rho,\one)^2\ .
  \end{align}
\item an \emph{entropy-dissipation} inequality with constant
  $\lambda>0$ if for all $\rho\in\PX$,
  \begin{align}
  \tag{EDI($\lambda$)}
  \label{eq:EDI}
    \cF(\rho) - \cF(\one) \leq \frac{1}{2\lambda}\cI(\rho)\ .
  \end{align}
\end{enumerate}  
\end{definition}

The following result relates these inequalities to $\cW$-geodesic
convexity of the entropy functional $\cF$. Recall the definition of
$\kappa_Q$ from \eqref{eq:kappa-q}.

\begin{theorem}\label{thm:functional-inequalities} 

  The following assertions hold.
  \begin{enumerate}
  \item If $\kappa_Q\in\R$, then $Q$
    satisfies $\cF\cW\cI(\kappa_Q)$.
  \item If $\kappa_Q>0$, then $Q$
    satisfies EDI$(\kappa_Q)$.
  \end{enumerate}
\end{theorem}

\begin{proof}
  The proof follows from similar arguments as the corresponding
  results in \cite{EM11} and uses the heuristics developed in the
  continuous case in \cite{OV00}. However, for the convenience of the
  reader we give a self-contained proof here.

  To prove (1), assume that $\kappa_Q \in \R$. It is sufficient
  to prove $\cF\cW\cI(\kappa_Q)$ for $\rho\in\cP_*(\cX)$. The inequality for
  general $\rho\in\PX$ then follows from an easy approximation
  argument. So fix $\rho\in\cP_*(\cX)$ and let
  $(\sigma_s)_{s\in[0,1]}$ be a constant speed geodesic connecting
  $\rho=\sigma_0$ to $\one=\sigma_1$.
  Using Proposition \ref{prop:Daneri-Savare}
  and rearanging terms in \eqref{eq:K-convex}, we find
  \begin{align}\label{eq:preFWI}
   \cF(\rho) - \cF(\one)\leq - \frac{1}{s}\Big(\cF(\sigma_s)-\cF(\sigma_0)\Big) - \frac{\kappa_Q}{2}(1-s)\cW(\rho,\one)^2\ .
  \end{align}
  Taking into account the smoothness of $(\sigma_t)$ near $t = 0$, we
  obtain the bound
  \begin{align*}
\lim_{s \to 0}\frac{|\cF(\sigma_s)-\cF(\sigma_0)|}{s} \leq \| \grad_\cW\cF(\sigma_0)\|_{\sigma_0}\;
\lim_{s \to 0}\frac{\cW(\sigma_s,\sigma_0)}{s}=\sqrt{\cI(\rho)}\cW(\rho,\one)\ .
  \end{align*}
  Hence, the first assertion of the theorem follows by passing to the
  limit $s\to 0$ in \eqref{eq:preFWI}.

  Let us now prove (2). Assume that $\kappa_Q>0$. From part
  (1) we know that $Q$ satisfies $\cF\cW\cI(\kappa_Q)$. From this we
  derive EDI$(\kappa_Q)$ by an application of Young's inequality:
\begin{align*}
x y \leq c x^2 +\frac{1}{4c}y^2\qquad \forall x,y\in\R\;, \ c>0\ ,
\end{align*}
  in which we set $x=\cW(\rho,\one),\ y=\sqrt{\cI(\rho)}$ and
  $c=\frac{\kappa_Q}{2}$.
\end{proof}

\section{Gromov-Hausdorff convergence}

In this final section we discuss the convergence of discrete
transportation metrics to the Wasserstein metric in a simple setting.

For $N \geq 2$ let $\TN = (\Z / N \Z)^d$ be the discrete torus in
dimension $d \geq 1$, which we regard as an approximation of the
continuous torus $\mathbb{T}^d = [0,1]^d$. We consider the matrix $Q :
\TN \times \TN \to \R$ defined by
\begin{align*}
Q(i,j) =  \left\{ \begin{array}{ll}
N^2,
 & \text{$i = j$},\\
-2 d N^2,
 & \text{$|i - j| =1$},\\
N^2,
 & \text{otherwise}.\end{array} \right.
\end{align*}
Note that the entries are scaled in such a way that the associated
discrete Laplacian $\Delta_N$ defined by \eqref{eq:discrete-Laplace}
approximates the Laplacian $\Delta$ on $\T^d$. Fix $0 < \theta \leq 1$
and let $\cWN$ denote the discrete transportation metric corresponding
to a weight function $\theta_m$.

The following result has already been announced in the introduction.

\begin{theorem}\label{thm:GH}
  Let $d \geq 1$ and $0 < m \leq 2$.
  Then the metric spaces $(\PTN,\cW_N)$ converge to
  $(\cP(\mathbb{T}^d),W_2)$ in the sense of Gromov-Hausdorff as $N 
  \to \infty$.
\end{theorem}

Since this result can be proved by following the argument in
\cite{GM12}, we shall not give a complete proof here. Instead, let us
point out the three crucial properties of $\theta = \theta_m$ which allow
us to reapply the argument from \cite{GM12}:
\begin{enumerate}
\item $\theta(t,t) = t$ for all $t \geq 0$.
\item $\theta$ is concave.
\item For all $a, b > 0$ we have \begin{align*}
\frac{1}{\tilde\theta(a,b)} - \frac{1}{\theta(a,b)} 
  \leq \frac{(b-a)^2}{ab}\frac{1}{\tilde\theta(a,b)}\;,
\end{align*}
where $\tilde\theta$ denotes the {harmonic mean} defined by 
$
\tilde\theta(a,b):=\frac{2ab}{a+b}
$.
\end{enumerate}
Assumption $(1)$ is clearly necessary, since it is already
checked at the formal level that the metrics may converge to a
different limit if $(1)$ is violated.  The concavity $(2)$ is used in 
the proof in \cite{GM12} to ensure that the discrete heat semigroup $(P_N(t))_{t \geq
  0}$ contracts the distance:
\begin{align*}
 \cWN\big(P_N(t) \rho_0, P_N(t) \rho_1\big) \leq 
 \cWN(\rho_0, \rho_1)\;.
\end{align*}
This estimate is used in regularisation arguments.  Finally, the
estimate $(3)$ enters the argument, since it turns out to be easier to
compare the Wasserstein metric to the discrete transportation metric
$\tWN$, which is defined using the harmonic mean. The estimate $(3)$
can then be combined with a regularisation argument, which shows that
the difference between $\cWN$ and $\tWN$ becomes negligible as $N$
gets large, provided that the weight function satisfies $(3)$.

Let us note that these properties are satisfied for $\theta_m$ with $0
< m \leq 2$. Indeed, (1) is obvious, and concavity of $\theta$ for $0
< m \leq 2$ has been proved in Lemma \ref{lem:theta-m}. Moreover,
since $\theta_{-1}$ is the harmonic mean, it follows from the
monotonicity of $\theta_m$ in $m$ (proved in Lemma \ref{lem:theta-m})
that it suffices to check (3) for $m =2$. In this case an elementary
computation shows that (3) indeed holds.

\bibliographystyle{plain}
\bibliography{ricci}

 \end{document}